\documentclass[]{amsart}
\usepackage{latexsym}
\usepackage{amsmath}
\usepackage{amsfonts}
\usepackage{amsthm}
\usepackage{amssymb, bbm}
\usepackage{amsrefs}
\usepackage[T1]{fontenc}

\title[quasiconvexity]{Two characterizations of quasiconvexity}
\author{W\l{}odzimierz Fechner}
\address{Institute of Mathematics, Lodz University of Technology, al. Politechniki 8, 93-590 \L\'od\'z, Poland}
\email{wlodzimierz.fechner@p.lodz.pl}
\newtheorem{thm}{Theorem}
\newtheorem{cor}{Corollary}
\theoremstyle{remark}
\newtheorem{rem}{Remark}
\newtheorem{ex}{Example}
\theoremstyle{definition}

\newcommand{\R}{\mathbb{R}}
\newcommand{\Q}{\mathbb{Q}}

\newcommand{\eps}{\varepsilon}


\begin{document}

\begin{abstract}
We present two characterisations of quasiconvexity for radially semicontinuous mappings defined on a convex subset of a real linear space. As applications, we obtain an extension of Sion's minimax theorem and two characterisations of quasiconvex risk measures.
\end{abstract}

\maketitle 

\section{Introduction}

Throughout this paper, we assume that $X$ is a real linear space, $D\subset X$ is a nonempty convex set and $\overline{\R}=\R \cup \{-\infty, +\infty\}$. 
For $x , y \in X$, by the segment $[x,y]$ we mean the set $\{tx +(1-t)y : t \in [0,1] \}$. On a segment, we will use the usual metric topology. Open subsegments will be denoted by $]u,v[\subset [x,y]$, to make a distinction between pairs of points $(u,v) \in X \times X$. We apply a similar convention for intervals on the real line.
Besides the topology, the segment $[x,y]$ has a natural total order. With these notations, the sets $[x,y]$ and $[y,x]$ contain the same elements and have identical topologies and reversed orders.

A map $f\colon D \to \overline{\R}$ is called \emph{quasiconvex} if 
\begin{equation}
\label{qc}
\forall_{x, y \in D}\forall_{t \in [0,1]}\, f(tx+ (1-t)y)\leq \max\{ f(x), f(y) \}
\end{equation}
or equivalently:
$$
\forall_{x, y \in D}\forall_{z \in [x,y]}\, f(z)\leq \max\{ f(x), f(y) \}.
$$
Mapping $f\colon D \to \overline{\R}$ is termed \emph{quasiconcave} if $-f$ is quasiconvex. 
We speak about strict quasiconvexity or strict quasiconcavity if the inequality sign in the respective definition is sharp whenever $x \neq y$ and $t\notin\{0,1\}$.
For a comprehensive reference to the theory of quasiconvex and quasiconcave functions the reader is referred to the classical monograph by Roberts and Varberg \cite{RV}.

If $Y$ is a metric space, then a function $f\colon Y \to \overline{\R}$ is \emph{upper semicontinuous} if
$$\forall_{x_0\in Y} \limsup_{x\to x_0}f(x) \leq f(x_0) $$ and $f$ is \emph{lower semicontinuous} if $-f$ is upper semicontinuous.
We refer the reader to the book by \L ojasiewicz \cite{Lo} for properties of semicontinuous functions that we will need later on.

A map $f\colon D \to \overline{\R}$ is called \emph{radially upper or lower semicontinuous} if it is upper or lower semicontinuous, respectively on each segment $[x,y] \subset D$. 
If $X$ is a normed space and $f$ is upper or lower semicontinuous, then a fortiori is radially upper or lower semicontinuous, respectively.

\section{Main results}

We aim to characterize quasiconvexity under minimal regularity assumptions.
This is done in Corollaries \ref{c2} and \ref{c4} below.
In our first theorem, we will prove a modification of a theorem of Dar\'{o}czy and P\'{a}les \cite{DP}*{Theorem 2.1}, where they dealt with nonconvex functions. The same result, but with a different proof can be found in a preprint by Leonetti \cite{Le}. We will use some ideas of their proofs and make some additional observations to obtain a more detailed description of the set on which a given mapping is not quasiconvex. In Remark \ref{r1} later on we propose a rephrasing and a slight extension of \cite{DP}*{Theorem 2.1}.

For a fixed function $f\colon D \to \overline{\R}$ we will denote
$$ T := \{ (x, y, z) \in D \times D \times D : z \in ]x,y[ \, \textrm{ and } \,f(z)> \max\{ f(x), f(y) \}  \}$$
and $$T(x,y) = \{ z \in ]x,y[ : (x,y, z) \in T\}$$ for $x, y  \in D$.

\begin{thm}\label{t1}
Suppose that $f\colon D \to \overline{\R}$ is a radially lower semicontinuous function. Then for every $x, y  \in D$ there exists an at most countable set $A=A(x,y)$ (possibly empty)  and the points $u_i, v_i \in [x,y],  i \in A$ such that:
\begin{enumerate}
	\item[(a)] $$\forall_{i \in A} u_i\neq v_i$$
	\item[(b)] $$\forall_{i, j \in A}\, [ i\neq j \, \Rightarrow \, ]u_i,v_i[\cap ]u_j,v_j[= \varnothing],$$
	\item[(c)] $$T(x,y) = \bigcup_{i \in  A}  ]u_i, v_i[,$$ 
	\item[(d)] $$\forall_{i \in A}  \, T(u_i,v_i)= ]u_i,v_i[  .$$
\end{enumerate}
\end{thm}
\begin{proof}
Let $x, y  \in D$. If $T(x,y)= \varnothing$, then the properties (a)-(d) hold trivially with $A=\varnothing$.

Suppose $T(x,y)\neq \varnothing$. Then $x \neq y$, $f(x)<+ \infty$,  $f(y)<+\infty$ and
$$
T(x,y)=\{z \in ]x,y[: f(z)> f(x) \}   \cap \{z \in ]x,y[: f(z)> f(y) \}.   
$$
Since $f$ is  radially lower semicontinuous, it follows that the set $T(x,y)$ is open in $]x,y[$, so there exists a family  $]u_i,v_i[\subset ]x,y[$,\, $i\in A $, where $\mathrm{card}(A)\leq \aleph_0$, satisfying the conditions (a)-(c).

It remains to prove (d). By definition $T(u_i,v_i)\subseteq ]u_i,v_i[$. Suppose that there exists $z \in ]u_i,v_i[ \setminus T(u_i,v_i)$. Then $f(z)\leq f(u_i)$ or $f(z)\leq f(v_i)$. Suppose  $f(z)\leq f(u_i)$. Since $z \in  T(x,y)$, it follows
$$f(x)<f(z)\leq f(u_i) \quad \textrm{and}\quad f(y)<f(z)\leq f(u_i) ,$$
implying $u_i \in T(x,y)$; a contradiction. Similarly, $f(z)\leq f(v_i)$ would imply $v_i \in T(x,y)$; a contradiction again. It follows $]u_i,v_i[ \subseteq T(u_i,v_i)$ and so, (d) holds too.
\end{proof}

\begin{cor}\label{c1}
Assume that $f\colon D \to \overline{\R}$ is a radially lower semicontinuous function. Then, either $f$ is quasiconvex, or 
$$\exists_{x, y \in D}\forall_{t \in ]0,1[}\, f(tx+ (1-t)y)> \max\{ f(x), f(y) \}.$$
\end{cor}
\begin{proof}
By Theorem \ref{t1}, if for each $x, y \in D$ one has $A = \varnothing$, then $T= \varnothing$, which implies that $f$ is quasiconvex. Otherwise, if for some $x, y \in D$ the set $T(x,y)$ is nonempty, then
by point (d) of Theorem \ref{t1} the set $T$ contains a subset of the form  $$\{u_i\}\times\{ v_i\}\times ]u_i,v_i[$$ with  some $u_i, v_i \in D$ such that $u_i\neq v_i$, which with the settings $x=u_i$ and $y = v_i$ proves the assertion.
\end{proof}

The next corollary follows immediately from the first one and it is an analogue of \cite{DP}*{Corollary 2.3}, see also \cite{Le}*{Theorem 1}. This is our first characterization of quasiconvexity.
Namely, for a radially lower semicontinuous function $f$, if for every pair $x, y \in D$ there is at least one number $t\in ]0,1[$ such that inequality in formula \eqref{qc} is satisfied, then it holds for all $t \in [0,1]$,  which means that $f$ is quasiconvex.

\begin{cor}\label{c2}
Assume that $f\colon D \to \overline{\R}$ is a radially lower semicontinuous function. 
then $f$ is quasiconvex on $D$ if and only if
$$\forall_{x, y \in D}\exists_{z \in ]x,y[}\, f(z) \leq \max\{f(x), f(y) \}.$$

\end{cor}

\begin{rem}\label{r1}
Utilizing some ideas of Theorem \ref{t1} one can rewrite \cite{DP}*{Theorem 2.1} of Dar\'{o}czy and P\'{a}les to obtain a reformulation that contains more information about the set of points of nonconvexity. Since the modification is straightforward, then we will omit the details. For a radially lower semicontinuous map $f\colon D \to \overline{\R}$ we introduce
$$ T_{conv} := \{ (x, y, t) \in D \times D \times [0,1] :  f(tx + (1-t)y)>  tf(x) + (1-t) f(y)  \}$$
and $$T_{conv}(x,y) = \{ t \in [0,1] : (x,y, t) \in T_{conv}\}$$ for $x, y  \in D$.
Then, analogous statements to (a)-(d) of Theorem \ref{t1} hold true. 

It is also interesting to note that many of the results of Kuhn \cites{K0, K} concerning $t$-convex functions or $(s,t)$-convex functions can be reformulated in the language of the set $T_{conv}$. In particular, in \cite{K0}*{Theorem 1K} it is assumed that $I$ is a real interval and $f\colon I \to \R\cup \{-\infty\}$ is an arbitrary function. This statement can be reformulated as follows: if there exists some $t \in ]0,1[$ such that for each $x, y \in I$ one has $(x,y, t)\in T_{conv}$, then for every $x, y \in I$ the set $T_{conv}(x,y)$ contains all rational points of $[0,1]$.
\end{rem}

The next example shows that our Theorem \ref{t1} and \cite{DP}*{Theorem 2.1} of Dar\'{o}czy and P\'{a}les are incomparable in general. 

\begin{ex}
Every monotone function $f\colon D \to \R$ with $D \subseteq \R$ is quasiconvex and quasiconcave simultaneously. For such mapping $f$ one has $T = \varnothing$. On the other hand, if $f$ is strictly concave, then $$T_{conv} =\{ (x, y, t) \in D : x\neq y  , \,  t\in ]0,1[\}.$$ Therefore, the sets $T$ and $T_{conv}$ can be fairly different.
\end{ex}

Next, we will show that the lower semicontinuity assumption cannot be dropped.

\begin{ex}
If $f= \mathbbm{1}_{\Q}$, i.e. the indicator function of the rationals, then
$$T = \{(x, y, z ) \in \R^3 : x, y \in \R\setminus \Q \, \textrm{ and } \, z \in [x,y] \cap \Q  \}.$$
In particular, $T$ contains no subset of the form  $$\{u_i\}\times\{ v_i\}\times ]u_i,v_i[$$
and therefore the assertion of Corollary \ref{c1} does not hold for function $f$. Clearly, $f$ is not lower semicontinuous.

A positive and nontrivial example can be obtained similarly with the aid of the complement of the Cantor set $C\subset [0,1]$. Indeed, if $f= \mathbbm{1}_{[0,1]\setminus C}$, then $f$ is lower semicontinuous, which means that Theorem \ref{t1} is in power. Therefore, there exists an infinite family of open intervals $]u_i,v_i[$ summing up to $[0,1]\setminus C$ that together with $f$ satisfies the second part of the assertion of Corollary \ref{c1}.
\end{ex}

\medskip

In our next theorem, we will deal with local quasiconvexity for upper quasicontinuous mappings defined on a real interval. Roughly speaking, we aim to show that under some assumptions, if a function is not quasiconvex, then it is locally quasiconcave around some point. Such a statement would provide an analogue to a theorem by P\'{a}les \cite{P}*{Theorem 2} for local convexity on the real line. We will obtain a stronger result, which is not surprising since the property of not being quasiconvex is stronger than the property of not being convex. 
Next, we derive a corollary for radially upper semicontinuous maps acting on a convex subset of a linear space. 

Assume that $I\subseteq \R$ is a real interval and $p\in I$ is an interior point of $I$. We say that $f\colon I \to \overline{\R}$ is \emph{locally quasiconvex at point $p$}, if
$$
\exists_{\delta>0}\forall_{x \in ]p-\delta,p[}\forall_{y \in ]p,p+\delta[} \, f(p) \leq \max\{f(x), f(y)\}.
$$
If the inequality is strict, then we speak about strictly locally quasiconvex maps. Locally quasiconcave and strictly locally quasiconcave mappings are understood in an obvious way.

It turns out that the notion of local quasiconvexity is closely related to some property of local maxima of the function in question. Obviously, every monotone function is quasiconvex and quasiconcave, as well. On the other hand, if a function has a strict local maximum at $p$, then it must not be quasiconvex at $p$. We will prove that a kind of the converse is also true under the upper semicontinuity assumption.

\begin{thm}\label{t2}
Assume that $f\colon I \to \overline{\R}$ is an upper semicontinuous function. Then, either $f$ is quasiconvex on $I$, or there exist points $p, q \in \mathrm{int}I$ such that:
\begin{enumerate}
	\item[(a)] $f(p) = f(q)$,
	\item[(b)] $f$ attains local maxima at $p$ and $q$,
	\item[(c)] the local maximum of $p$ is strict from the left and the local maximum of $q$ is strict from the right;
\end{enumerate}
in particular, in the latter case $f$ is locally strictly quasiconcave at $p$ and $q$.
\end{thm}
\begin{proof}
Assume that $f$ is not quasiconvex. So, there exist $x_0, y_0 \in I$ and $t_0 \in ]0,1[$ such that
$$f(t_0x_0+ (1-t_0)y_0)> \max\{ f(x_0), f(y_0) \}.$$
Without loss of generality, we can assume that $x_0<y_0$, since equality $x_0=y_0$ is impossible.
Let us consider the set of points of the interval $[x_0,y_0]$, where $f$ attains its supremum on $[x_0,y_0]$:
$$H:= \{ z \in  [x_0,y_0] : f(z) = \sup\{ f(t) : t \in  ]x_0,y_0[ \}\, \}.$$
Clearly, $x_0, y_0 \notin H$. Moreover, by the upper semicontinuity, the set $H$ is nonempty and compact. 
Take $p=\min H$ and $q = \max H$.  We have by definition
$$
\forall_{x \in ]x_0, y_0[} \, f(x) \leq f(p) = f(q)
$$
which shows (a) and (b). We also have $x_0 < p\leq q <y_0$ and
$$
\forall_{x \in ]x_0, p[} \, f(x) < f(p) , \quad \forall_{y \in ]q, y_0[} \, f(y) < f(q).
$$
We thus proved point (c). This implies that $f$ is locally strictly quasiconcave at $p$ and at $q$.
\end{proof}

\begin{rem}
In the foregoing theorem, it is not excluded that $p=q$.
\end{rem}

We are at the point to derive from Theorem \ref{t2} our next characterization of quasiconvexity. We begin with an immediate consequence of the preceding theorem in a one-dimensional case.

\begin{cor}\label{c3}
Assume that $f\colon I \to \overline{\R}$ is an upper semicontinuous function which has the property that if $f$ has a local maximum that is strict from one side, then it must not attain a local maximum strict from the other side.
Then $f$ is quasiconvex on $I$.
\end{cor}

\begin{cor}\label{c4}
Assume that $f\colon D \to \overline{\R}$ is a radially upper semicontinuous function which has the property that for every point $p\in D$ and every segment $[x,y]\subseteq D$ such that $p \in [x,y]$, if $f|_{[x,y]}$ has a local maximum at $p$ which is strict from one direction, then $f|_{[x,y]}$ has no local maximum which is strict from another direction. Then $f$ is quasiconvex on $D$.
\end{cor}
\begin{proof}
Fix $x, y \in D$ and apply Corollary \ref{c3} for $f$ restricted to the interval $[x,y]$.
\end{proof}

\begin{rem}
The property of $f$ given in Corollary \ref{c4} is in general weaker than the local quasiconcavity of $f$ at each point. Corollary \ref{c4} implies that in the class of upper semicontinuous functions, they are equivalent (and both are equivalent to quasiconvexity of $f$ on $I$). A similar equivalence is true in the case of radially upper semicontinuous functions defined on $D\subseteq X$.
\end{rem}

\begin{ex}
Again, the Cantor set $C\subset [0,1]$ provides us an illustrative example. Take $f= \mathbbm{1}_{C}$, then $f$ is upper semicontinuous and it is neither quasiconvex nor quasiconcave. Moreover, the set $H$ spoken of in Theorem \ref{t2} is equal to the Cantor set $C$. Mapping $f$ serves as an example of a function which satisfies the assumptions of Theorem \ref{t2}. At the same time, $f$ has no strict local maximum and some of its maxima are strict from one side. Therefore the assertion of this theorem and of Corollaries \ref{c3} and \ref{c4} cannot be strengthened considerably.  
\end{ex}

\section{Applications}

In the last section, we provide interpretations of our findings for a few mathematical objects, for which the assumption of quasiconvexity or quasiconcavity appears.

\subsection{Sion's minimax theorem}
One of generalizations of the von Neumann's minimax theorem is Sion's theorem obtained by Sion \cite{S}. This result is of particular interest to us since it joins assumptions of lower semicontinuity with quasiconvexity. Later, new proofs and several generalizations have been published by several authors. An elementary proof is due to Komiya \cite{Ko}.

\begin{thm}[Sion \cite{S}]
Assume that $X$ is a compact convex subset of a linear topological space and $Y$ is a convex subset of a linear topological space. Let $f \colon X\times Y\to \R$
be such that
\begin{itemize}
	\item[(i)] for each $x \in X$ $f(x, \cdot )$ is upper semicontinuous and quasiconcave on $Y$,
	\item[(ii)] for each $y \in Y$ $f (\cdot, y )$ is lower semicontinuous and quasiconvex on $X$.
\end{itemize}
Then $$\min\limits_{x \in X} \sup\limits_{y \in Y} f(x,y) =  \sup\limits_{y \in Y}\min\limits_{x \in X} f(x,y).$$
\end{thm}

Utilizing Corollary \ref{c2} we obtain the following extension of Sion's minimax theorem. 

\begin{cor}
Assume that $X$ is a compact convex subset of a linear topological space and $Y$ is a convex subset of a linear topological space. Let $f \colon X\times Y\to \R$
be such that
\begin{itemize}
	\item[(i)] for each $x \in X$ $f(x, \cdot )$ is upper semicontinuous and 
	$$\forall_{y, z \in Y}\exists_{u \in ]y,z[}\, f(x, u) \geq \min\{f(x, y), f(x, z) \},$$
	\item[(ii)] for each $y \in Y$ $f( \cdot, y )$ is lower semicontinuous and 
	$$\forall_{x, z \in X}\exists_{u \in ]x,z[}\, f(u, y) \leq \max\{f(x, y), f(z, y) \}.$$
\end{itemize}
Then $$\min\limits_{x \in X} \sup\limits_{y \in Y} f(x,y) =  \sup\limits_{y \in Y}\min\limits_{x \in X} f(x,y).$$
\end{cor}

\subsection{Quasiconvex risk measures} 
Assume that $\mathcal{X}$ is the set of all financial positions (i.e. a vector space that satisfies $L^{\infty}(\mathbb{P}) \subseteq \mathcal{X} \subseteq L^{0}(\mathbb{P})$ with a probability measure $\mathbb{P}$). A risk measure is a map $\rho\colon \mathcal{X}\to\R$. Classical axioms of the  risk measure are: normalization, i.e. $\rho(0)=0$, decreasing monotonicity and cash additivity (called also translativity), i.e.
$$\rho(X+c\cdot\mathbbm{1}) = \rho(X) - c, \quad X \in \mathcal{X}, c \in \R.$$
 
We speak about coherent risk measure if additionally positive homogeneity and subadditivity are assumed. The two last assumptions together are significantly stronger than convexity of $\rho$.
Convexity of the risk measure often was interpreted as follows: diversification of our portfolio will never increase the risk. However, as it is observed by Cerreia-Vioglio, Maccheroni, Marinacci, Montrucchio in \cite{C} a proper mathematical formulation of this principle is the quasiconvexity, rather than convexity of the risk measure. Moreover, cash additivity should be replaced by inequality.

Using Corollary \ref{c2} we can formulate the following characterization of quasiconvex risk measures.

\begin{cor} 
If a risk measure $\rho\colon \mathcal{X}\to\R$ is radially lower semicontinuous and has the following property:

{for each positions $X, Y \in \mathcal{X}$ there exists a diversification, i.e. a position strictly between $X$ and $Y$, say $Z=tX + (1-t)Y$ with some $t\in ]0,1[$ such that $Z$ is no more risky than each $X$ and $Y$,}

then $\rho$ is a quasiconvex risk measure. 
\end{cor}

Corollary \ref{c4} leads to a similar characterization.

\begin{cor} 
If a risk measure $\rho\colon \mathcal{X}\to\R$ is radially upper semicontinuous and has the following property:

{for each positions $X, Y \in \mathcal{X}$, if there exists a position  between $X$ and $Y$, say $Z=tX + (1-t)Y$ with some $t\in ]0,1[$ such that $Z$ locally strictly maximizes risk among positions $sX + (1-s)Y$ with $t-\delta <s\leq t$ for some $\delta>0$, then there is no position $Z=uX + (1-u)Y$ that  locally strictly maximizes risk among positions of the form  $rX + (1-r)Y$ with $u \leq r < u + \eps$ for some $\eps >0$,}

then $\rho$ is a quasiconvex risk measure. 
\end{cor}

\section{Conclusions}

In this article, we have presented two novel characterizations of quasiconvexity for radially semicontinuous mappings defined on convex subsets of real linear spaces. Our first characterization, based on a modification of a result by Dar\'{o}czy and P\'{a}les, provides a description of the set where quasiconvexity fails for radially lower semicontinuous functions. The second characterization, derived from an analysis of local maxima, applies to radially upper semicontinuous functions and connects the absence of specific patterns of strict local maxima on segments to the property of quasiconvexity.

As applications of our theoretical findings, we have extended Sion's minimax theorem by weakening the standard quasiconvexity and quasiconcavity assumptions to conditions based on the existence of points satisfying certain inequality relations on segments. Furthermore, we have provided two characterizations of quasiconvex risk measures in terms of their radial semicontinuity and specific properties related to diversification and local risk maximization.












\end{document}